\pdfoutput=1
\documentclass[11pt,oneside]{amsart}
\usepackage[paperheight=279mm,paperwidth=18cm,textheight=26cm,textwidth=14cm]{geometry}
\usepackage[T1]{fontenc}
\usepackage[utf8]{inputenc}
\usepackage[unicode]{hyperref}
\hypersetup{
bookmarks=true,
colorlinks=true,
citecolor=[rgb]{0,0,0.5},
linkcolor=[rgb]{0,0,0.5},
urlcolor=[rgb]{0,0,0.75},
pdfpagemode=UseNone,
pdfstartview=FitH,
pdfdisplaydoctitle=true,
pdftitle={Martingale square function},
pdfauthor={Dennis Wollgast, Pavel Zorin-Kranich},
pdflang=en-US
}

\usepackage{amsmath}
\usepackage{amssymb}
\usepackage{amsthm}
\usepackage{stmaryrd}
\usepackage{xspace}
\usepackage{mathtools}

\usepackage[style=alphabetic,maxalphanames=4]{biblatex}
\numberwithin{equation}{section}
\theoremstyle{plain}
\newtheorem{theorem}{Theorem}[section]
\newtheorem{proposition}[theorem]{Proposition}
\newtheorem{lemma}[theorem]{Lemma}

\newtheorem{definition}[theorem]{Definition}

\theoremstyle{remark}
\newtheorem{remark}[theorem]{Remark}

\DeclarePairedDelimiter\abs{\lvert}{\rvert}
\DeclarePairedDelimiter\norm{\lVert}{\rVert}

\def\E{\mathbb{E}}
\DeclarePairedDelimiterXPP\EE[1]{\E}{\lparen}{\rparen}{}{#1} 

\DeclarePairedDelimiterX\Set[1]\{\}{%

#1
}

\def\C{\mathbb{C}}
\newcommand*{\N}{\mathbb{N}}
\newcommand*{\R}{\mathbb{R}}

\newcommand{\dif}{\mathop{}\!\mathrm{d}} 

\def\PZdefchar#1{
  \expandafter\def\csname frak#1\endcsname{\mathfrak{#1}}
  \expandafter\def\csname rm#1\endcsname{\mathrm{#1}}
  \expandafter\def\csname bb#1\endcsname{\mathbb{#1}}
  \expandafter\def\csname bf#1\endcsname{\mathbf{#1}}
  \expandafter\def\csname scr#1\endcsname{\mathcal{#1}}
  \expandafter\def\csname cal#1\endcsname{\mathcal{#1}}}
\def\PZdefloop#1{\ifx#1\PZdefloop\else\PZdefchar#1\expandafter\PZdefloop\fi}
\PZdefloop abcdefghijklmnopqrstuvwxyzABCDEFGHIJKLMNOPQRSTUVWXYZ\PZdefloop

\title[Martingale square function]{Weighted Davis inequalities\\ for martingale square functions}
\author[D.~Wollgast]{Dennis Wollgast}
\author[P.~Zorin-Kranich]{Pavel Zorin-Kranich}
\address{Mathematical Institute\\ University of Bonn}
\email{pzorin@uni-bonn.de}

\makeatletter
\@namedef{subjclassname@2020}{%
  \textup{2020} Mathematics Subject Classification}
\makeatother

\begin{document}
\begin{abstract}
For a Hilbert space valued martingale $(f_{n})$ and an adapted sequence of positive random variables $(w_{n})$, we show the weighted Davis type inequality
\[
\E \Bigl( \abs{f_{0}} w_{0} + \frac{1}{4} \sum_{n=1}^{N} \frac{\abs{df_{n}}^{2}}{f^{*}_{n}} w_{n} \Bigr)
\leq
\E( f^{*}_{N} w^{*}_{N}).
\]
This inequality is sharp and implies several results about the martingale square function.
We also obtain a variant of this inequality for martingales with values in uniformly convex Banach spaces.
\end{abstract}
\maketitle

\section{Introduction}
Throughout this article, $(f_{n})_{n\in\N}$ denotes a martingale on a filtered probability space $(\Omega,(\calF_{n})_{n\in\N})$ with values in a Banach space $(X,\abs{\cdot})$.
A \emph{weight} is a positive random variable on $\Omega$.
We denote martingale differences and running maxima by
\[
df_{n} = f_{n}-f_{n-1},
\quad
f^{*}_{n} := \max_{n'\leq n} \abs{f_{n'}},
\quad
w^{*}_{n} := \max_{n'\leq n} w_{n'}.
\]
We begin with the Hilbert space valued case of our main result (Theorem~\ref{thm:S<M}).
\begin{theorem}
\label{thm:Hilbert}
Let $(f_{n})_{n\in\N}$ be a martingale with values in a Hilbert space $(X=H,\abs{\cdot})$.
Let $(w_{n})_{n\in\N}$ be an adapted sequence of weights (that need not be a martingale).
Then, for every $N\in\N$, we have
\begin{equation}
\label{eq:Hilbert-L1}
\E \Bigl( \abs{f_{0}} + \frac{1}{3} \sum_{n=1}^{N} \frac{\abs{df_{n}}^{2}}{f^{*}_{n}} \Bigr)
\leq
\E( f^{*}_{N} )
\end{equation}
and
\begin{equation}
\label{eq:Hilbert-A1}
\E \Bigl( \abs{f_{0}} w_{0} + \frac{1}{4} \sum_{n=1}^{N} \frac{\abs{df_{n}}^{2}}{f^{*}_{n}} w_{n} \Bigr)
\leq
\E( f^{*}_{N} w^{*}_{N}).
\end{equation}
\end{theorem}
A quantity similar to the left-hand side of \eqref{eq:Hilbert-L1}, but with $f^{*}_{N}$ in place of $f^{*}_{n}$ and hence smaller, appeared in \cite[\textsection 3]{MR331502}.

In order to relate our result to the usual martingale square function
\[
Sf := \bigl( \sum_{n=1}^{N} \abs{df_{n}}^{2} \bigr)^{1/2},
\]
we note that, by Hölder's inequality,
\begin{equation}
\label{eq:Sf-Holder}
\E Sf
\leq
\E \Bigl( (f^{*}_{N})^{1/2} \bigl( \sum_{n=1}^{N} \frac{\abs{df_{n}}^{2}}{f^{*}_{n}} \bigr)^{1/2} \Bigr)
\leq
\Bigl( \E f^{*}_{N} \Bigr)^{1/2} \Bigl( \E \sum_{n=1}^{N} \frac{\abs{df_{n}}^{2}}{f^{*}_{n}} \Bigr)^{1/2}.
\end{equation}
By one of the Burkholder--Davis--Gundy inequalties \cite{MR0268966}, we have $\E f^{*}_{N} \leq C \E Sf$ for martingales with $f_{0}=0$ (the optimal value of $C$ does not seem to be known; the value $C=\sqrt{10}$ was obtained in \cite[\nopp II.2.8]{MR0448538}).
Assuming that both sides are finite, this implies
\[
\E Sf
\leq
C \E \sum_{n=1}^{N} \frac{\abs{df_{n}}^{2}}{f^{*}_{n}}
\]
with the same constant $C$.
Thus, we see that \eqref{eq:Hilbert-L1} adds a new equivalence to the $L^{1}$ Burkholder--Davis--Gundy inequalties.

The proof of Theorem~\ref{thm:Hilbert} is based on Burkholder's proof of the Davis inequality for the square function with the sharp constant \cite{MR1859027} and its weighted extension by Os\k{e}kowski \cite{MR3567926}.
Note, however, that the weights in the latter article are assumed to be continuous in time, so that it does not yield weighted estimates in discrete time.
The estimate \eqref{eq:Hilbert-A1} is instead motivated by \cite{MR3688518}, where the Davis inequality for the martingale maximal function was proved with a similar combination of weights $(w,w^{*})$.
Such weighted inequalities go back to \cite{MR0284802}, see also \cite[Theorem 3.2.3]{MR3617205} for a martingale version.

\subsection{Sharpness of the constants}
Both estimates \eqref{eq:Hilbert-L1} and \eqref{eq:Hilbert-A1} are sharp, in the sense that the constants $1/3$ and $1/4$ cannot be replaced by any larger constants.

The sharpness of \eqref{eq:Hilbert-L1} is due to the fact that it implies the sharp version of the Davis inequality for the expectation of the martingale square function \cite{MR1859027}.
To see this, for notational simplicity, suppose $f_{0}=0$.
By \eqref{eq:Sf-Holder} and \eqref{eq:Hilbert-L1}, we have
\begin{align*}
\E Sf
\leq
\sqrt{3} \E f^{*}_{N}.
\end{align*}
Since the constant $\sqrt{3}$ is the smallest possible in this inequality \cite[\textsection 5]{MR1859027}, also the constant in \eqref{eq:Hilbert-L1} is optimal.

The sharpness of \eqref{eq:Hilbert-A1} is proved in Section~\ref{sec:sharpness-of-4}.

\subsection{Consequences of the weighted estimate}
\label{sec:extrapolation}
Here, we show how Theorem~\ref{thm:Hilbert} can be used to recover a number of known inequalities.

Let $r\in [1,2]$, $w$ be an integrable weight, $w_{n} = \E(w|\calF_{n})$, $f^{*} = f^{*}_{\infty}$, and $w^{*}=w^{*}_{\infty}$.
For simplicity, we again assume $f_{0}=0$.
By Hölder's inequality and \eqref{eq:Hilbert-L1}, we obtain
\begin{equation}
\label{eq:S-w}
\begin{split}
\E \Bigl( (Sf)^{r} \cdot w \Bigr)
&\leq
\E \Bigl( (f^{*})^{(2-r)r/2} \bigl( \sum_{n=1}^{N} \frac{\abs{df_{n}}^{2}}{f^{*}_{n}} (f^{*}_{n})^{r-1} \bigr)^{r/2} w \Bigr)
\\ &\leq
\Bigl( \E (f^{*})^{r} w \Bigr)^{1-r/2}
\Bigl( \E \bigl( \sum_{n=1}^{N} \frac{\abs{df_{n}}^{2}}{f^{*}_{n}} (f^{*}_{n})^{r-1} w \bigr) \Bigr)^{r/2}
\\ &=
\Bigl( \E (f^{*})^{r} w \Bigr)^{1-r/2}
\Bigl( \E \bigl( \sum_{n=1}^{N} \frac{\abs{df_{n}}^{2}}{f^{*}_{n}} (f^{*}_{n})^{r-1} w_{n} \bigr) \Bigr)^{r/2}
\\ &\leq
2^{r} \Bigl( \E ( f^{*} )^{r} w \Bigr)^{1-r/2} \Bigl( \E ( f^{*} )^{r} w^{*} \Bigr)^{r/2}.
\end{split}
\end{equation}
If we estimate $w\leq w^{*}$ in the first term on the right-hand side, we recover a version of the main result in \cite{MR3567926}.
Our version has a worse constant, but does not require the weights to be continuous in time.

Recall that the \emph{$A_{1}$ characteristic} of a weight $w$ is the smallest constant $[w]_{A_{1}}$ such that $w^{*} \leq [w]_{A_{1}} w$.
As a direct consequence of the estimate \eqref{eq:S-w}, we obtain the estimate
\begin{equation}
\label{eq:1-weight-A1}
\E \Bigl( \bigl( \sum_{n=1}^{N} \abs{df_{n}}^{2} \bigr)^{1/2} w \Bigr)
\leq
2 [w]_{A_{1}}^{1/2} \E( f^{*} w)
\end{equation}
for $A_{1}$ weights $w$.
This improves the main result of \cite{MR3738183}, where a similar estimate (with $\sqrt{5}$ in place of $2$) was proved for dyadic martingales.
In view of \cite[Theorem 1.3]{BrzOse2021}, it seems unlikely that the $A_{1}$ characteristic in \eqref{eq:1-weight-A1} can be replaced by a function of any $A_{p}$ characteristic with $p>1$, although for dyadic martingales even the $A_{\infty}$ characteristic suffices \cite[Theorem 2]{MR352854}.

By a version of the Rubio de Francia argument, one can deduce further $L^{p}$ weighted estimates from \eqref{eq:S-w} (with $r=1$).
Let $p \in (1,\infty)$, $w$ a weight, and $\tilde{w} = w^{-p'/p}$ the dual weight, where $p'$ denotes the Hölder conjugate that is determined by $1/p+1/p'=1$.
Let also
\[
Mh := \sup_{n \in \N} \abs{\E(h | \calF_{n})}
\]
denote the martingale maximal operator.
Then, for any function $u$, by \eqref{eq:S-w} and Hölder's inequality, we obtain
\begin{equation}
\label{eq:Sf-u-w}
\begin{split}
\E ( Sf \cdot u \cdot w )
&\leq 2
\E ( Mf \cdot M(uw) )^{1/2} \cdot \E ( Mf \cdot (uw) )^{1/2}
\\ &\leq 2
(\E (Mf)^{p} w)^{1/p} (\E (M(uw))^{p'} w^{-p'/p})^{1/(2p')} (\E u^{p'} w)^{1/(2p')}.
\end{split}
\end{equation}
By definition of the operator norm, we have
\begin{align*}
(\E (M(uw))^{p'} w^{-p'/p})^{1/p'}
&\leq
\norm{ M }_{L^{p'}(\tilde{w}) \to L^{p'}(\tilde{w})}
(\E (uw)^{p'} w^{-p'/p})^{1/p'}
\\ &=
\norm{ M }_{ L^{p'}(\tilde{w}) \to L^{p'}(\tilde{w})}
(\E u^{p'} w)^{1/p'}.
\end{align*}
Substituting this into \eqref{eq:Sf-u-w}, we obtain
\[
\E ( Sf \cdot u \cdot w )
\leq 2
\norm{ M }_{ L^{p'}(\tilde{w}) \to L^{p'}(\tilde{w})}^{1/2}
(\E (Mf)^{p} w)^{1/p} (\E u^{p'} w)^{1/p'}.
\]
By duality, this implies
\begin{equation}
\label{eq:S-Lp-w}
\norm{ Sf }_{L^{p}(w)}
\leq 2
\norm{ M }_{ L^{p'}(\tilde{w}) \to L^{p'}(\tilde{w})}^{1/2}
\norm{ f^{*} }_{L^{p}(w)}.
\end{equation}
In the case $w\equiv 1$, using Doob's maximal inequality \cite[Theorem 3.2.2]{MR3617205}, this recovers the following version of the martingale square function inequality, which matches \cite[Theorem 3.2]{MR365692}:
\[
\norm{ Sf }_{L^{p}} \leq 2 \sqrt{p} p' \norm{ f }_{L^{p}}.
\]
More generally, an $A_{p}$ weighted BDG inequality can be obtained from \eqref{eq:S-Lp-w} using the $A_{p'}$ weighted martingale maximal inequality proved in \cite{arxiv:1607.06319}.

Another Rubio de Francia type extrapolation argument, see \cite[Appendix A]{arxiv:2106.07281}, can be used to deduce UMD Banach space valued estimates from either \eqref{eq:Hilbert-A1} or \eqref{eq:S-w}.
This recovers one of the estimates in \cite[Theorem 1.1]{MR4181372} (the other direction similarly follows from the weighted estimate in \cite{MR3688518}).

\section{Uniformly convex Banach spaces}
In this section, we recall a few facts about uniformly convex Banach spaces that are relevant to the Banach space valued version of Theorem~\ref{thm:Hilbert}, Theorem~\ref{thm:S<M}.

\begin{definition}
Let $q \in [2,\infty)$.
A Banach space $(X,\abs{\cdot})$ is called \emph{$q$-uniformly convex} if there exists $\delta>0$ such that, for every $x,y\in X$, we have
\begin{equation}
\label{eq:q-unif-convex}
\abs[\big]{\frac{x+y}{2}}^{q} + \delta \abs[\big]{\frac{x-y}{2}}^{q}
\leq
\frac{\abs{x}^{q} + \abs{y}^{q}}{2}.
\end{equation}
\end{definition}

We will use a different (but equivalent) characterization of uniform convexity, in terms of the convex function $\phi : X\to\R_{\geq 0}$, $\phi(x) = \abs{x}^{q}$ and its directional derivative at point $x$ in direction $h$, given by
\begin{equation}
\label{eq:directional-derivative}
\phi'(x)h := \lim_{t\to 0, t>0} \frac{\phi(x+th)-\phi(x)}{t}.
\end{equation}
Convexity of $\phi$ is equivalent to the right-hand side of \eqref{eq:directional-derivative} being an increasing function of $t$ for fixed $x,h$.
By the triangle inequality and Taylor's formula, we have
\[
\abs{\abs{x+h}^{q} - \abs{x}^{q}}
\leq
(\abs{x}+\abs{h})^{q} - \abs{x}^{q}
\leq
q\abs{x}^{q-1}\abs{h} + o_{\abs{h}\to 0}(\abs{h}).
\]
Therefore, the quotient on the right-hand side of \eqref{eq:directional-derivative} is bounded from below.
Hence, the limit \eqref{eq:directional-derivative} exists, and we have
\begin{equation}
\label{eq:phi'-bd}
\abs{\phi'(x)h} \leq q \abs{x}^{q-1} \abs{h}.
\end{equation}
Moreover, for every $x\in X$, the function $h \mapsto \phi'(x)h$ is convex, which follows directly from convexity of $\phi$.

\begin{lemma}
A Banach space $(X,\abs{\cdot})$ is uniformly convex if and only if, for every $x,h\in X$, we have
\begin{equation}
\label{eq:convexity}
\abs{x+h}^{q} \geq \abs{x}^{q} + \phi'(x) h + \tilde{\delta}\abs{h}^{q}.
\end{equation}
Moreover, the largest $\delta,\tilde{\delta}$ for which \eqref{eq:q-unif-convex} and \eqref{eq:convexity} hold satisfy
\begin{equation}
\label{eq:equivalence-of-unif-convexity-const}
\frac{\delta}{2^{q-1}-1} \leq \tilde{\delta} \leq \delta.
\end{equation}
\end{lemma}

The estimate \eqref{eq:convexity} can only hold with $\tilde{\delta} \leq 1$ (unless $X$ is $0$-dimensional), as can be seen by taking $x=0$.
When $X$ is a Hilbert space, we can take $q=2$ and $\delta=1$ in \eqref{eq:q-unif-convex} by the parallelogram identity and $\tilde{\delta}=1$ in \eqref{eq:convexity} by \eqref{eq:equivalence-of-unif-convexity-const}.

\begin{proof}
Clearly, the sets of $\delta$ and $\tilde{\delta}$ for which \eqref{eq:q-unif-convex} and \eqref{eq:convexity} hold are closed, so we may consider the largest such $\delta$ and $\tilde{\delta}$.

To see the first inequality in \eqref{eq:equivalence-of-unif-convexity-const}, let $\calC$ be the set of all constants $c\geq 0$ such that, for every $x,h\in X$, we have
\[
\phi(x+h) \geq \phi(x) + \phi'(x) h + c \phi(h).
\]
By convexity of $\phi$, we have $0 \in \calC$.

Let $c\in\calC$.
For any $x,h \in X$, using the uniform convexity assumption \eqref{eq:q-unif-convex} with $y=x+h$, we obtain
\[
\abs{x+h/2}^{q} + \delta \abs{h/2}^{q} \leq (\abs{x}^{q}+\abs{x+h}^{q})/2.
\]
By the definition of $c\in\calC$, it follows that
\[
\abs{x}^{q} + \phi'(x) h/2 + c \abs{h/2}^{q} + \delta \abs{h/2}^{q} \leq (\abs{x}^{q}+\abs{x+h}^{q})/2.
\]
Rearranging this inequality, we obtain
\[
\abs{x}^{q} + \phi'(x) h + 2 c \abs{h/2}^{q} + 2 \delta \abs{h/2}^{q} \leq \abs{x+h}^{q}.
\]
Therefore, $2^{1-q}(c+\delta) \in \calC$.
Since $c\in\calC$ was arbitrary, this implies
\[
\sup \calC \geq \delta/(2^{q-1}-1).
\]

To see the second inequality in \eqref{eq:equivalence-of-unif-convexity-const}, note that convexity of $h \mapsto \phi'(x)h$ implies $\phi'(x)h + \phi'(x)(-h) \geq 0$.
Applying \eqref{eq:convexity} with $(z,h)$ and $(z,h)$, we obtain
\begin{align*}
\abs{z+h}^{q} + \abs{z-h}^{q}
&\geq
2\abs{z}^{q} + \phi'(z) h + \phi'(z)(-h) + \tilde{\delta}\abs{h}^{q} + \tilde{\delta}\abs{-h}^{q}
\\ &\geq
2\abs{z}^{q} + 2 \tilde{\delta}\abs{h}^{q}.
\end{align*}
With the change of variables $x=z+h$, $y=z-h$, we obtain \eqref{eq:q-unif-convex} with $\tilde{\delta}$ in place of $\delta$.
\end{proof}

With the characterization of uniform convexity in \eqref{eq:convexity} at hand, we can finally state our main result in full generality.

\begin{theorem}
\label{thm:S<M}
For every $q \in [2,\infty)$, there exists $\gamma=\gamma(q) \in \R_{>0}$ such that the following holds.

Let $(X,\abs{\cdot})$ be a Banach space such that \eqref{eq:convexity} holds.
Let $(f_{n})_{n\in\N}$ be a martingale with values in $X$, and $(w_{n})_{n\in\N}$ an adapted sequence of weights.
Then,
\begin{equation}
\label{eq:S<M}
\E \Bigl( \gamma \abs{f_{0}} w_{0} + \tilde{\delta} \sum_{n=1}^{\infty} \frac{\abs{df_{n}}^{q}}{(f^{*}_{n})^{q-1}} w_{n} \Bigr)
\leq
\gamma \E( f^{*} w^{*})
\end{equation}
In the case $q=2$, we can take $\gamma=4$.
In the case $q=2$, $\tilde{\delta}=1$, and $w_{n}=1$ for all $n\in\N$, we can take $\gamma=3$.
\end{theorem}

In order to see that the linear dependence on $\tilde{\delta}$ in \eqref{eq:S<M} is optimal, we can apply this inequality with $w_{n} = (f^{*}_{n})^{q-1}$, followed by Doob's maximal inequality, which gives the estimate
\[
\E \Bigl( \gamma \abs{f_{0}}^{q} + \tilde{\delta} \sum_{n=1}^{\infty} \abs{df_{n}}^{q} \Bigr)
\leq
\gamma \E ( f^{*} )^{q}
\leq
(q')^{q} \gamma \sup_{n\in\N} \E \abs{f_{n}}^{q}.
\]
By \cite[Theorem 10.6]{MR3617459}, linear dependence on $\tilde{\delta}$ is optimal in this inequality, and hence the same holds for \eqref{eq:S<M}.

Similarly as in Section~\ref{sec:extrapolation}, Theorem~\ref{thm:S<M} implies several weighted extensions of the martingale cotype inequality \cite[Theorem 10.59]{MR3617459}.
We omit the details.

\section{The Bellman function}
The proof of Theorem~\ref{thm:S<M} is based on the Bellman function technique; we refer to the books \cite{MR2964297,VasVol_Bellman_book} for other instances of this technique.
The particular Bellman function that we use here goes back to \cite{MR1859027}; the first weighted version of it was introduced in \cite{MR3567926}.
For $x\in X$ and $y,m,v \in \R_{\geq 0}$ with $\abs{x} \leq m$, we define
\[
U(x,y,m,v) := \tilde\delta y - \frac{\abs{x}^{q}+(\gamma-1)m^{q}}{m^{q-1}} v,
\]
where $\gamma = \gamma(q)$ will be chosen later.
The main feature of this function is the following concavity property.
\begin{proposition}
\label{prop:concavity}
Suppose that $\gamma$ is sufficiently large depending on $q$ (see \eqref{eq:gamma-general} for the precise condition).
Let $(X,\abs{\cdot})$ be a Banach space such that \eqref{eq:convexity} holds.
Then, for any $x,h \in X$ and $y,m,w,v \in \R_{\geq 0}$ with $\abs{x}\leq m$, we have
\begin{equation}
\label{eq:inductive-step}
U(x+h,y+\frac{w\abs{h}^{q}}{(\abs{x+h}\vee m)^{q-1}},\abs{x+h}\vee m,v \vee w)
\leq
U(x,y,m,v) - \frac{v \phi'(x) h}{m^{q-1}}.
\end{equation}
\end{proposition}

\begin{proof}[Proof of Theorem~\ref{thm:S<M} assuming Proposition~\ref{prop:concavity}]
Using \eqref{eq:inductive-step} with
\[
x = f_{n},
\quad
y = \tilde{S}_{n} := \gamma\abs{f_{0}} w_{0}
+ \tilde{\delta} \sum_{j=1}^{n} \frac{\abs{df_{j}}^{q}}{(f^{*}_{j})^{q-1}} w_{j},
\quad
m = f^{*}_{n},
\]
\[
w=w_{n},
\quad
v=w^{*}_{n},
\quad
h = df_{n+1},
\]
we obtain
\begin{equation}
U(f_{n+1},\tilde{S}_{n+1},f^{*}_{n+1},w^{*}_{n+1})
\leq
U(f_{n},\tilde{S}_{n},f^{*}_{n},w^{*}_{n}) - \frac{\phi'(f_{n}) df_{n+1}}{(f^{*}_{n})^{q-1}} w_{n}^{*}.
\end{equation}
By convexity of $h \mapsto \phi'(x)h$, we have
\[
\E (\frac{\phi'(f_{n}) df_{n+1}}{(f^{*}_{n})^{q-1}} w_{n}^{*} | \calF_{n} )
=
\frac{w_{n}^{*}}{(f^{*}_{n})^{q-1}} \E ( \phi'(f_{n}) df_{n+1} | \calF_{n} )
\geq 0.
\]
Taking expectations, we obtain
\[
\E U(f_{n+1},\tilde{S}_{n+1},f^{*}_{n+1},w^{*}_{n+1})
\leq
\E U(f_{n},\tilde{S}_{n},f^{*}_{n},w^{*}_{n}).
\]
Iterating this inequality, we obtain
\begin{multline*}
\E \Bigl( \gamma \abs{f_{0}} w_{0} + \tilde{\delta} \sum_{n=1}^{N} \frac{\abs{df_{n}}^{q}}{(f^{*}_{n})^{q-1}} w_{n} - \gamma f^{*}_{N} w^{*}_{N} \Bigr)
\\ \leq
\E U(f_{N},\tilde{S}_{N},f^{*}_{N},w^{*}_{N})
\leq
\E U(f_{0},\tilde{S}_{0},f^{*}_{0},w^{*}_{0})
=
0.
\qedhere
\end{multline*}
\end{proof}

\begin{remark}
The above proof in fact shows the pathwise inequality
\[
\gamma \abs{f_{0}} w_{0} + \tilde{\delta} \sum_{n=1}^{N} \frac{\abs{df_{n}}^{q}}{(f^{*}_{n})^{q-1}} w_{n}
\leq
\gamma f^{*}_{N} w^{*}_{N}
- \sum_{n=1}^{N}
\frac{\phi'(f_{n}) df_{n+1}}{(f^{*}_{n})^{q-1}} w_{n}^{*}.
\]
This can be used to improve the first part of \cite[Theorem 1.1]{MR3322322}.
For simplicity, consider the scalar case $X=\C$ (so that $q=2$ and $\tilde{\delta}=1$) with $f_{0}=0$ and $w_{n}=1$.
The above inequality then simplifies to
\[
\sum_{n=1}^{N} \frac{\abs{df_{n}}^{q}}{f^{*}_{n}}
\leq
3 f^{*}_{N}
- \sum_{n=1}^{N} \frac{2 f_{n} df_{n+1}}{f^{*}_{n}}.
\]
Using \eqref{eq:Sf-Holder}, the above inequality, and concavity of the function $x\mapsto x^{1/2}$, we obtain
\begin{align*}
S_{N}f
&\leq
(f^{*}_{N})^{1/2} \Bigl( \sum_{n=1}^{N} \frac{\abs{df_{n}}^{2}}{f^{*}_{n}} \Bigr)^{1/2}
\\ &\leq
(f^{*}_{N})^{1/2} \Bigl( 3 f^{*}_{N} - \sum_{n=1}^{N} \frac{2 f_{n} df_{n+1}}{f^{*}_{n}} \Bigr)^{1/2}
\\ &\leq
(f^{*}_{N})^{1/2} \Bigl( (3 f^{*}_{N})^{1/2} - \frac{1}{2}(3 f^{*}_{N})^{-1/2} \sum_{n=1}^{N} \frac{2 f_{n} df_{n+1}}{f^{*}_{n}} \Bigr)
\\ &=
\sqrt{3} f^{*}_{N} - \sum_{n=1}^{N} \frac{f_{n} df_{n+1}}{\sqrt{3} f^{*}_{n}}.
\end{align*}
\end{remark}

\begin{proof}[Proof of Proposition~\ref{prop:concavity}]
If $\abs{x+h}\leq m$, then
\begin{align*}
\MoveEqLeft
U(x+h,y+\frac{w\abs{h}^{q}}{(\abs{x+h}\vee m)^{q-1}},\abs{x+h}\vee m,v \vee w)
\\ &=
\tilde{\delta}(y+\frac{w\abs{h}^{q}}{m^{q-1}})
-\frac{\abs{x+h}^{q}+(\gamma-1)m^{q}}{m^{q-1}} (v\vee w)
\\ &\leq
\tilde{\delta}(y+\frac{w\abs{h}^{q}}{m^{q-1}})
-
\frac{\abs{x}^{q} + \phi'(x) h + \tilde{\delta}\abs{h}^{q}+(\gamma-1)m^{q}}{m^{q-1}} (v\vee w)
\\ &\leq
\tilde{\delta} y
-
\frac{\abs{x}^{q} + \phi'(x) h +(\gamma-1)m^{q}}{m^{q-1}} (v\vee w)
\\ &\leq
\tilde{\delta} y -
\frac{\abs{x}^{q} + \phi'(x) h + (\gamma-1)m^{q}}{m^{q-1}} v
\\ &=
U(x,y,m,w,v) - \frac{\phi'(x) h}{m^{q-1}} v.
\end{align*}
In the last inequality, we used
\begin{equation}
\label{eq:aux-gamma-1}
\abs{\phi'(x)h}
\leq
q \abs{x}^{q-1} \abs{h}
\leq
q \abs{x}^{q-1} (\abs{x} + m)
\leq
\abs{x}^{q} + (2q-1) m^{q}
\leq
\abs{x}^{q} + (\gamma-1) m^{q},
\end{equation}
which holds provided that $\gamma \geq 2q$.

If $\abs{x+h}>m$, then we need to show
\begin{multline*}
\tilde{\delta}(y+\frac{w\abs{h}^{q}}{\abs{x+h}^{q-1}})
- \frac{\abs{x+h}^{q} + (\gamma-1) \abs{x+h}^{q}}{\abs{x+h}^{q-1}} (v\vee w)
\\ \leq
\tilde{\delta}y - \frac{\abs{x}^{q} + (\gamma-1)m^{q}}{m^{q-1}} v
- \frac{\phi'(x) h}{m^{q-1}} v.
\end{multline*}
This is equivalent to
\begin{equation}
\label{eq:2}
\frac{\tilde{\delta}\abs{h}^{q}w - \gamma \abs{x+h}^{q} (v\vee w)}{\abs{x+h}^{q-1}}
\leq
\frac{-\abs{x}^{q} v-(\gamma-1)m^{q}v}{m^{q-1}}
- \frac{\phi'(x) h}{m^{q-1}} v.
\end{equation}
Assuming that $\gamma\geq 2^{q}\tilde{\delta}$, we have
\begin{equation}
\label{eq:aux-gamma-2}
\tilde{\delta} \abs{h}^{q}
\leq
\tilde{\delta} (\abs{x+h} + \abs{x})^{q}
\leq
\tilde{\delta} (2 \abs{x+h})^{q}
\leq
\gamma \abs{x+h}^{q},
\end{equation}
and it follows that the left-hand side of \eqref{eq:2} is
\[
\leq (v\vee w) \frac{\tilde{\delta}\abs{h}^{q}-\gamma \abs{x+h}^{q}}{\abs{x+h}^{q-1}}
\leq v \frac{\tilde{\delta}\abs{h}^{q}-\gamma \abs{x+h}^{q}}{\abs{x+h}^{q-1}}.
\]
Hence, it suffices to show
\begin{equation}\label{eq:1}
\frac{\tilde{\delta}\abs{h}^{q}-\gamma \abs{x+h}^{q}}{\abs{x+h}^{q-1}}
\leq
\frac{-\abs{x}^{q}-(\gamma-1)m^{q}}{m^{q-1}} - \frac{\phi'(x) h}{m^{q-1}}.
\end{equation}
Let $t := \abs{x+h}/m > 1$ and $\tilde{t} := \abs{h}/m$.
Note that $\abs{t-\tilde{t}} = \abs{\abs{x+h}-\abs{h}}/m \leq \abs{h}/m \leq 1$.
We will show \eqref{eq:1} in two different ways, depending on the values of $t,\tilde{t}$.

\textbf{Estimate 1.}
By \eqref{eq:convexity}, the inequality \eqref{eq:1} will follow from
\[
\frac{\tilde{\delta}\abs{h}^{q}-\gamma \abs{x+h}^{q}}{\abs{x+h}^{q-1}}
\leq
\frac{-\abs{x+h}^{q}+\tilde{\delta}\abs{h}^{q}-(\gamma-1)m^{q}}{m^{q-1}}.
\]
This is equivalent to
\[
\tilde{\delta}\tilde{t}^{q}/t^{q-1}-\gamma t
\leq
-t^{q}+\tilde{\delta}\tilde{t}^{q}-(\gamma-1).
\]
Equivalently,
\[
t^{q}-\gamma t + (\gamma-1)
\leq
\tilde{\delta}\tilde{t}^{-q} (1-1/t^{q-1}).
\]
so this would follow from
\[
\gamma
\geq
\frac{1}{t-1} \bigl( t^{q}-1 - \tilde{\delta} \tilde{t}^{q}(1-1/t^{q-1}) \bigr).
\]

\textbf{Estimate 2.}
The inequality \eqref{eq:1} is implied by
\[
\frac{\tilde{\delta}\abs{h}^{q}}{\abs{x+h}^{q-1}}
+
\frac{\abs{x}^{q}+(\gamma-1)m^{q}}{m^{q-1}}
+
\frac{\abs{\phi'(x) h}}{m^{q-1}}
\leq
\gamma \abs{x+h}.
\]
This is equivalent to
\[
\frac{\tilde{\delta}\tilde{t}^{q}}{t^{q-1}}
+
\frac{\abs{x}^{q}}{m^{q}}
+
(\gamma-1)
+
\frac{\abs{\phi'(x) h}}{m^{q}}
\leq
\gamma t.
\]
Using \eqref{eq:phi'-bd}, we see that the left-hand side is bounded by
\[
\frac{\tilde{\delta} \tilde{t}^{q}}{t^{q-1}}
+
1
+
(\gamma-1)
+
\frac{q \abs{x}^{q-1} \cdot \abs{h}}{m^{q}}
\leq
\frac{\tilde{\delta} \tilde{t}^{q}}{t^{q-1}}
+
\gamma
+
q \tilde{t}.
\]
Hence, it suffices to assume
\[
\frac{\tilde{\delta} \tilde{t}^{q}}{t^{q-1}}
+
\gamma
+
q (t+1)
\leq
\gamma t,
\]
or, in other words,
\[
\gamma
\geq
\frac{1}{t-1}
\bigl( \frac{\tilde{\delta} \tilde{t}^{q}}{t^{q-1}} + q \tilde{t} \bigr).
\]
Combining the two estimates, we see that \eqref{eq:1} holds provided that
\begin{equation}
\label{eq:gamma-complicated-lower-bd}
\gamma \geq
\sup_{t>1, \abs{t-\tilde{t}} \leq 1} \frac{1}{t-1}
\min \bigl( t^{q}-1 - \tilde{\delta} \tilde{t}^{q}(1-1/t^{q-1}),
\frac{\tilde{\delta} \tilde{t}^{q}}{t^{q-1}} + q \tilde{t} \bigr).
\end{equation}
In order to obtain a more easily computable bound, we estimate
\[
RHS\eqref{eq:gamma-complicated-lower-bd}
\leq
\sup_{t\geq 1} \sup_{K\geq 0} \frac{1}{t-1} \min\bigl( t^{q}-1 - K(1-1/t^{q-1}),
K/t^{q-1} + q (t+1) \bigr).
\]
Since we are taking the minimum of an increasing and a decreasing function in $K$, the supremum over $K$ is achieved for the value of $K$ for which these functions take equal values, or for $K=0$ if the latter value in negative.
Hence, substituting $K=\max(t^{q}-1-q(t+1),0)$, we obtain
\[
RHS\eqref{eq:gamma-complicated-lower-bd}
\leq
\sup_{t\geq 1} \frac{1}{t-1} \bigl( t^{q}-1 - \max(t^{q}-1-q(t+1),0)(1-1/t^{q-1}) \bigr),
\]
The function $t \mapsto t^{q}-1-q(t+1)$ is strictly monotonically increasing, so there is a unique solution $t_{0}$ to $t_{0}^{q}-1-q(t_{0}+1) = 0$.
The supremum is then assumed for $t=t_{0}$, since
\[
\frac{\dif}{\dif t} \frac{t^{q}-1}{t-1} = \frac{(q-1)t^{q}+1-qt^{q-1}}{(t-1)^{2}} \geq 0
\]
by the AMGM inequality, and
\[
\frac{\dif}{\dif t} \frac{1}{t-1} \bigl( q(t+1) + \frac{t^{q}-1}{t^{q-1}} \bigr)
=
- \frac{2q}{(t-1)^{2}} - \frac{t^{q-2}}{(t^{q}-t^{q-1})^{2}} (t^{q}-qt+(q+1)) \leq 0.
\]
Hence,
\[
RHS\eqref{eq:gamma-complicated-lower-bd}
\leq
\frac{q(t_{0}+1)}{t_{0}-1}.
\]
Collecting the conditions on $\gamma$ in the proof, we see that it suffices to assume
\begin{equation}
\label{eq:gamma-general}
\gamma \geq \max ( \frac{q(t_{0}+1)}{t_{0}-1}, 2q, 2^{q}).
\end{equation}
For $q=2$, we have $t_{0}=3$, so we can take $\gamma = 4$.

In the case $v=w=1$, we do not use \eqref{eq:aux-gamma-1} and \eqref{eq:aux-gamma-2}, so the only condition on $\gamma$ is given by \eqref{eq:gamma-complicated-lower-bd}.
If we additionally assume $q=2$ and $\tilde{\delta} = 1$, that condition can be further simplified in the same way as in \cite{MR3567926}.
Namely, it suffices to ensure
\[
\gamma \geq
\sup_{t>1, \abs{t-\tilde{t}} \leq 1} \frac{1}{t-1}
\bigl( t^{2}-1 - \tilde{t}^{2}(1-1/t) \bigr).
\]
The supremum in $\tilde{t}$ is assumed for $\tilde{t} = (t-1)$, so this condition becomes
\[
\gamma \geq
\sup_{t>1} \frac{1}{t-1}
\bigl( t^{2}-1 - (t-1)^{2}(1-1/t) \bigr)
=
\sup_{t>1} t+1 - (t-1)^{2}/t
=
\sup_{t>1} 3 - 1/t
=
3.
\]
This is the bound used in \eqref{eq:Hilbert-L1}.
\end{proof}

\section{Optimality}
\label{sec:sharpness-of-4}
In this section, we show that the inequality \eqref{eq:Hilbert-A1} fails if $1/4$ is replaced by any larger number, already if the weights constitute a (positive) martingale.

Let $\Omega = \N_{\geq 1}$ with the filtration $\calF_{n}$ such that the $n$-th $\sigma$-algebra $\calF_{n}$ is generated by the atoms $\Set{1},\dotsc,\Set{n}$.
Let $k\in\R_{>0}$ be arbitrary.
The measure on $(\Omega,\calF=\vee_{n\in\N} \calF_{n})$ is given by $\mu(\Set{\omega}) = k(k+1)^{-\omega}$.
The martingale and the weights are given by
\[
f_{n}(\omega) =
\begin{cases}
(-1)^{\omega+1} \frac{k+2}{k}, & \omega\leq n,\\
(-1)^{n}, & \omega>n,
\end{cases}
\quad
w_{n}(\omega) =
\begin{cases}
0, & \omega \leq n,\\
(k+1)^{n}, & \omega > n.
\end{cases}
\]
Note that both these processes are indeed martingales.
Their running maxima are given by
\[
f^{*}_{n}(\omega) =
\begin{cases}
\frac{k+2}{k}, & \omega\leq n,\\
1, & \omega>n,
\end{cases}
\quad
w^{*}_{n}(\omega) =
\begin{cases}
(k+1)^{\omega-1} & \omega \leq n,\\
(k+1)^{n}, & \omega > n.
\end{cases}
\]
Now, we compute both sides of \eqref{eq:Hilbert-A1}:
\[
\E \sum_{n \leq N} \frac{\abs{df_{n}}^{2}}{f^{*}_{n}} w_{n}
=
\sum_{n \leq N} \mu(\N_{>n}) \frac{2^{2}}{1} (k+1)^{n}
=
\sum_{n \leq N} (k+1)^{-n} \frac{2^{2}}{1} (k+1)^{n}
=
4N,
\]
and
\begin{align*}
\E (f^{*}_{N} w^{*}_{N})
&=
\sum_{\omega\leq N} \mu(\Set{\omega}) \frac{k+2}{k} (k+1)^{\omega-1}
+
\mu(\N_{>N}) \cdot 1 \cdot (k+1)^{N}
\\ &=
\sum_{\omega\leq N} k(k+1)^{-\omega} \frac{k+2}{k} (k+1)^{\omega-1}
+
(k+1)^{-N} \cdot 1 \cdot (k+1)^{N}
\\ &=
N \frac{k+2}{k+1} + 1.
\end{align*}
Since $k$ and $N$ can be arbitrarily large, we see that the constant in \eqref{eq:Hilbert-A1} is optimal.

\printbibliography
\end{document}